\documentclass[11pt,spanish,english]{amsart}
\usepackage[T1]{fontenc}
\usepackage[latin9]{inputenc}
\usepackage{geometry}
\usepackage{color}

\usepackage{babel}
\addto\shorthandsspanish{\spanishdeactivate{~<>}}

\usepackage[normalem]{ulem}
\usepackage{placeins}
\usepackage{cancel}
\usepackage{units}
\usepackage{amsthm}
\usepackage{amstext}
\usepackage{amssymb}
\usepackage{setspace}
\usepackage{stackrel}
\usepackage[unicode=true,
 bookmarks=true,bookmarksnumbered=true,bookmarksopen=true,bookmarksopenlevel=3,
 breaklinks=false,pdfborder={0 0 0}]{hyperref}

\makeatletter

\numberwithin{equation}{section}
\numberwithin{figure}{section}
\numberwithin{table}{section}
\theoremstyle{plain}

\makeatother

\newcommand{\Z}{{\mathbb Z}}
\newcommand{\Q}{{\mathbb Q}}

\newcommand{\GEN}[1]{\langle #1 \rangle}
\newcommand{\PSL}{\text{PSL}}
\newcommand{\Tr}{\text{Tr}}

\newcommand{\var}{{\varepsilon}}

\newcommand{\diag}{{\rm diag}}

\newcommand{\V}{\mathrm{V}}
\newcommand{\B}{{\mathbb B}}

\newcommand{\MM}[2]{\left( #1 : #2 \right)}
\newcommand{\AMM}[2]{\left| #1 : #2 \right|}

\DeclareMathOperator{\Gal}{Gal}

\title{Zassenhaus Conjecture on torsion units holds \\ for $\PSL(2,p)$  with $p$ a Fermat or Mersenne prime}
\thanks{This research is partially supported by the European Commission under Grant 705112-ZC, by the Spanish Government under Grant MTM2016-77445-P with "Fondos FEDER" and, by Fundación Séneca of Murcia under Grant 19880/GERM/15.}

\author{Leo Margolis}
\author{Ángel del Río}
\author{Mariano Serrano}

\address{Departamento de Matemáticas, Universidad of Murcia. 30100 Murcia, Spain}

\email{leo.margolis@um.es, adelrio@um.es, mariano.serrano@um.es}
\date{\today}
\subjclass[2010] {16U60, 16S34}
\newcommand{\PP}[1]{{#1}'}

\newtheorem{theorem}{Theorem}[section]
\newtheorem{lemma}[theorem]{Lemma}
\newtheorem{proposition}[theorem]{Proposition}

\newtheorem{corollary}[theorem]{Corollary}

\begin{document}

\begin{abstract}
H.J. Zassenhaus conjectured that any unit of finite order in the integral group ring $\Z G$ of a finite group $G$ is conjugate in the rational group
algebra $\Q G$ to an element of the form $\pm g$ with $g\in G$. Though known for some series of solvable groups, the conjecture has been
proved only for thirteen non-abelian simple groups. We prove the Zassenhaus Conjecture for the groups $\PSL(2,p)$, where $p$ is a Fermat or Mersenne prime.
This increases the list of non-abelian simple groups for which the conjecture is known by probably infinitely many, but at least by $49$, groups. Our result is
an easy consequence of known results and our main theorem which states that the Zassenhaus Conjecture holds for a unit in $\Z\PSL(2,q)$ of order
coprime with $2q$, for some prime power $q$.
\end{abstract}

\maketitle
\section{Introduction}
One of the most famous open problems regarding the unit group of an integral group ring $\Z G$ of a finite group $G$ is the Zassenhaus Conjecture which was stated by H.J. Zassenhaus \cite{Zassenhaus}:

\begin{quote}
\textbf{Zassenhaus Conjecture\footnote{After this paper was submitted a metabelian counterexample to the Zassenhaus Conjecture was announced in \cite{EiseleMargolis}. Still no simple counterexample is known.}:} If $G$ is a finite group and $u$ is a unit of finite order in the integral group ring $\Z G$,
then there exists a unit $x$ in the rational group algebra $\Q G$ and an element $g \in G$ such that $x^{-1}ux = \pm  g$.
\end{quote}

If for a given $u$ such $x$ and $g$ exist, one says that $u$ and $\pm g$ are rationally conjugate.
The Zassenhaus Conjecture found much attention and was proved for many series of solvable groups, e.g. for nilpotent groups \cite{Weiss91}, groups possessing a
normal Sylow subgroup with abelian complement \cite{HertweckColloq} or cyclic-by-abelian groups \cite{CyclicByAbelian}.
Regarding non-solvable groups, however, the conjecture is only known for very few groups.
The proofs of the results for solvable groups mentioned above often argue by induction on the order of the group.
In this way one may assume that the conjecture holds for proper quotients of the original group.
The first step in a similar argument for non-solvable groups should consist in proving the conjecture for simple groups.
Although this has been studied by some authors, see e.g. \cite{LutharPassi, HertweckBrauer, HertweckA6, KonovalovM22, SalimA9A10, BachleMargolisEding07, AndreasMauricio}, the conjecture is still only known for exactly thirteen non-abelian simple groups all being isomorphic to some $\PSL(2,q)$ for some particular small prime power $q$
(see \cite[Theorem~C]{BachleMargolis4-PrimaryII} for an overview).
Our aim in this paper is to extend this knowledge by proving the following theorem.

\begin{theorem}\label{Main}
Let $G = \PSL(2,q)$ for some prime power $q$.  Then any torsion unit of $\Z G$ of order coprime with $2q$  is rationally conjugate to an element of $G$.
\end{theorem}

We prove this result employing a variation of a well known method which uses characters of a finite group $G$ to obtain restrictions on the possible torsion units in $\Z G$.
The idea of the method was introduced for ordinary characters by Luthar and Passi \cite{LutharPassi} and extended to Brauer characters by Hertweck  \cite{HertweckBrauer}. Today this method is often called
the HeLP (\textbf{He}rtweck\textbf{L}uthar\textbf{P}assi) Method.
In fact to prove our results we do not use the HeLP Method in the classical sense, since this would imply too many case distinctions. For this reason we vary the method in a way suitable for the character theory of $\PSL(2,q)$. Theorem~\ref{Main} can be regarded as a generalization of \cite[Theorem~1]{SylowPSL}.

As a direct application of Theorem~\ref{Main} and known facts about the units of $\Z\PSL(2,q)$ collected in Theorem~\ref{GroupProperties}, we obtain the result which gives name to this paper:

\begin{theorem}\label {FermatMersenne}
Let $p$ be a Fermat or Mersenne prime. Then the Zassenhaus Conjecture holds for $\PSL(2,p)$.
\end{theorem}

This result increases the number of simple groups for which the Zassenhaus Conjecture is known from thirteen to sixty-two: The groups $\PSL(2,q)$
with $q \in \{8, 9, 11, 13, 16, 19, 23, 25, 32 \}$  or one of the four known Fermat primes different from 3 or one of the forty-nine known Mersenne primes different
from 3 \cite{CaldwellMersennePrimes}.
Actually, Theorem~\ref{FermatMersenne} proves the conjecture for probably infinitely many simple groups because, based on heuristic evidences, it has been
conjectured that there are infinitely many Mersenne primes \cite{CaldwellHeuristics}.
Lenstra, Pomerance and Wagstaff have proposed independently a conjecture on the growth of the number of Mersenne primes smaller than a given integer
\cite{Pomerance,Wagstaff}.

It has been shown in \cite{AngelMariano} that a result as in Theorem \ref{Main} can not be achieved using solemnly the HeLP  Method if the unit has order $2p$,
where $2p$ is coprime with $q$ and $p$ a prime bigger than $3$. Looking on the orders of elements in $\PSL(2,q)$, cf.
Theorem~\ref{GroupProperties}, one should not
expect a better result for the Zassenhaus Conjecture for $\PSL(2,q)$ when applying only this method.
Thus, as so often in Arithmetics and Group Theory, the prime $2$ behaves very differently than the other primes.

We collect in Section~\ref{Preliminaries} the notation and known results which will be used during the proof of Theorem~\ref{Main}.
In Section~\ref{SectionNumberTheoreticalResult} we prove several number theoretical results which are essential for our arguments and introduce some more notation.
Finally in Section~\ref{SectionProofOfMainResult} we prove Theorem~\ref{Main}.

\section{Preliminaries}\label{Preliminaries}

Let $G$ be a finite group.
If $g\in G$, then $|g|$ denotes the order of $g$, the cyclic group generated by $g$ is denoted by $\langle g\rangle$
and $g^{G}$ denotes the conjugacy class of $g$ in $G$.
If $R$ is a ring then $RG$ denotes the group ring of $G$ with coefficients in $R$.
Denote by $\V(\Z G)$ the group of normalized units (i.e units of augmentation $1$) in $\Z G$.
As mentioned above, we say that two elements of $\Z G$ are rationally conjugate if they are conjugate in the units of $\Q G$.

The main notion to study rational conjugacy of torsion units in $\Z G$ are the so called partial augmentations. If $\alpha=\sum_{g\in G} \alpha_g g$ is an element of a group ring $\Z G$, with each $\alpha_g\in \Z$, then the partial augmentation of $\alpha$ at $g$ is defined as
	$$\var_{g}(\alpha)=\sum_{h\in g^{G}}\alpha_{h}.$$
The relevance of partial augmentations for the study of the Zassenhaus Conjecture is provided by a result of Marciniak, Ritter, Sehgal and Weiss.
The following theorem states this result and collects some known information about partial augmentations.

\begin{theorem}\label{Known}
Let $G$ be a finite group and let $u$ be an element of order $n$ in $\V(\Z G)$.
\begin{enumerate}
\item \cite[Theorem 2.5]{MarciniakRitterSehgalWeiss}\label{MRSW} $u$ is rationally conjugate to element in $G$ if and only if $\varepsilon_g(u^d) \geq 0$ for all $g \in G$ and all divisors $d$ of $n$.
\item\cite[Proposition 1.5.1]{EricAngel1} (Berman-Higman Theorem) If $u \neq 1$ then $\varepsilon_1(u) = 0$.
\item\cite[Theorem 2.3]{HertweckBrauer}\label{OrdersDivides} If $\varepsilon_g(u)\ne 0$ then $|g|$ divides $n$.
\item\cite[Theorem 3.2]{HertweckBrauer}\label{BrauerCharacterOnUnit}
Let $p$ be a prime not dividing $n$ and let $\chi$ be a $p$-Brauer character of $G$ associated to a modular representation $G \rightarrow M_m(k)$
for a suitable $p$-modular system $(K,R,k)$. Then $\chi$ extends to a $p$-Brauer character defined on the $p$-regular torsion units of $\Z G$, associated to
the natural algebra homomorphism $RG\rightarrow M_m(k)$. Moreover, if $g_1,...,g_k$ are representatives of the $p$-regular conjugacy classes of $G$ then
\begin{equation}\label{CharacterPA}
\chi(u) = \sum_{i = 1}^k \varepsilon_{g_i}(u)\chi(g_i).
\end{equation}
\end{enumerate}
\end{theorem}

We collect the group theoretical properties of $\PSL(2,q)$ and its integral group ring relevant for us.

\begin{theorem}\label{GroupProperties}
Let $G = \PSL(2,q)$ where $q = t^f$ for some prime $t$ and let $d = \gcd(2,q)$.
\begin{enumerate}
\item\cite[Hauptsatz 8.27]{Huppert1}\label{PSL2GroupProperties} The following properties hold.

\begin{itemize}
\item The order of $G$ is $(q-1)q(q+1)/d$.
\item The orders of elements in $G$ are exactly $t$ and the divisors of $(q+1)/d$ and $(q-1)/d$.
\item Two cyclic subgroups of $G$ are conjugate in $G$ if and only if they have the same order.
\item If $g,h\in G$ with $|g|$ coprime with $t$ and multiple of $|h|$ then $h$ is conjugate in $G$ to an element $h_1$ of $\GEN{g}$ and the only elements of $\GEN{g}$ conjugate  to $h$ in $G$ are $h_1$ and $h_1^{-1}$.
In particular a conjugacy class of elements of order coprime with $t$ is a real conjugacy class.
\end{itemize}
\item\label{EigenvaluesSymmetry} If $u$ is a torsion element of $\V(\Z G)$ of order coprime with $t$, $\zeta$ is root of unity in an arbitrary field $F$ and $\Theta$ is an $F$-representation of $G$ then $\zeta$ and $\zeta^{-1}$ have the same multiplicity as eigenvalues of $\Theta(u)$. This follows from \eqref{PSL2GroupProperties} and the formulas for multiplicities of eigenvalues of torsion units as presented in \cite[Section 4]{HertweckBrauer}.

\item \label{PSL2PrimePower} Let $u \in \V(\Z G)$ of order $n$.
\begin{itemize}
\item If $\gcd(n,q)=1$ then $G$ has an element of order $n$ \cite[Proposition~6.7]{HertweckBrauer}.
\item If $n$ is a prime power not divisible by $t$, then $u$ is rationally conjugate to an element of $G$ \cite[Theorem 1]{SylowPSL}.
\item If moreover $f = 1$ and $n$ is divisible by $t$, then $n = t$ and $u$ is also rationally conjugate to an element of $G$
\cite[Propositions~6.1, 6.3]{HertweckBrauer}.
\end{itemize}

\item \cite[Lemma 1.2]{SylowPSL}\label{PSL2ModularCharacters} Let $n$ be a positive integer coprime with $t$ and let $g \in G$ be an element of order $n$. There exists a primitive $n$-th root of unity $\alpha$ in a field of characteristic $t$ such that for every positive integer $m$, there is a $t$-modular representation $\Theta_{m}$ of $G$ of degree $1+2m$ such that
 	$$\Theta_{m}(g)\text{ is conjugate to } \diag\left(1,\alpha,\alpha^{-1},\alpha^{2},\alpha^{-2},\ldots,\alpha^{m},\alpha^{-m}\right).$$
We denote by $\psi_m$ the Brauer character associated with $\Theta_m$.
\end{enumerate}
\end{theorem}

As mentioned in the introduction, we actually do not use the HeLP Method
in its classical setting. We neither compute many inequalities involving traces as for example in the proofs of
\cite[Proposition~6.5]{HertweckBrauer} or \cite{KonovalovM22, SylowPSL}, since these formulas turn out to be too complicated in our setting.
Nor do we apply the standard equations obtained from character values on one side and
possible eigenvalues on the other side as e.g. in the proofs of \cite[Propositions 6.4, 6.7]{HertweckBrauer}, \cite{HertweckA6} or \cite[Lemma 2.2]{BachleMargolisEding07}, since there are too many
possibilities for these possible eigenvalues. Still this second strategy is closer to our approach.

\section{Number theoretical results} \label{SectionNumberTheoreticalResult}

In this section we prove two number theoretical results which are essential for our arguments and might be of independent interest.
Our first proof of Proposition~\ref{NT} below was very long.
We include a proof which was given to us by Hendrik Lenstra. We are very thankful to him for his simple and nice proof.

For a prime integer $p$ and a non-zero integer $n$ let $v_p(n)$ denote the valuation of $n$ at $p$, i.e. the maximal non-negative integer $m$ with $p^m\mid n$.
If, moreover, $n>0$ then $\zeta_n$ denotes a complex primitive $n$-th root of unity and $\Phi_n$ denotes the $n$-th cyclotomic polynomial, i.e. the minimal polynomial of $\zeta_n$ over $\Q$.

\begin{lemma}\label{LemaCyclotomicPolynomials}
	If $n$ and $m$ are positive integers and $p$ is a prime integer then $\Phi_{np^m}(\zeta_n)\in p\; \Z[\zeta_n]$.
\end{lemma}

\begin{proof}
We argue by induction on  $v_p(n)$.
Suppose first that $p\nmid n$ and let $S$ denote the set of primitive $p^m$-th roots of unity.
Then $\zeta_n\xi$ is a root of $\Phi_{np^m}(X)$ for every $\xi\in S$ and hence
$\prod_{\xi\in S}(X-\zeta_n\xi)$ divides $\Phi_{np^m}(X)$ in $\Z[\zeta_n][X]$.
Therefore
$$\Phi_{np^m}(\zeta_n)\in \prod_{\xi\in S}(\zeta_n-\zeta_n\xi)\;\Z[\zeta_n] =
\prod_{\xi\in S}(1-\xi)\;\Z[\zeta_n]=\Phi_{p^m}(1)\;\Z[\zeta_n]=p\;\Z[\zeta_n].$$
Suppose that $p \mid n$ and assume that the lemma holds with $n$ replaced by $\frac{n}{p}$.
Then $\Phi_{np^{m-1}}(\zeta_n^p)=\Phi_{\frac{n}{p}p^m}(\zeta_n^p)\in p\; \Z[\zeta_n^p]\subseteq p\; \Z[\zeta_n]$.
As $\Phi_{np^m}(X)=\Phi_{np^{m-1}}(X^p)$ and $\zeta_n^{p}$ is a primitive $\frac{n}{p}$-th root of unity, we have $\Phi_{np^m}(\zeta_n)=\Phi_{np^{m-1}}(\zeta_n^p)\in p\; \Z[\zeta_n]$.
\end{proof}

\begin{proposition}\label{NT}
Let $n$ be a positive integer.
Let $A_0,A_1,\dots,A_{n-1}$ be integers and for every positive integer $i$ set
\[\omega_i = \sum_{j=0}^{n-1} A_j \zeta_n^{ij}.\]
Let $d$ be a divisor of $n$ such that $\omega_{\frac{d}{q}}=0$ for every prime power $q$ dividing $d$ with $q\ne 1$.
Then $\omega_d\in d\;\Z[\zeta_n]$.
\end{proposition}

\begin{proof}
	Let $k=\frac{n}{d}$ and consider the polynomial $f(X)=\sum_{j=0}^{n-1} A_jX^j$.
	We can take $\zeta_k=\zeta_n^d$, so that $\omega_d=f(\zeta_k)$.
	By hypothesis, for every prime $p$ and every positive integer $m$ with $p^m$ dividing $d$ we have $f(\zeta_{kp^m})=0$, or equivalently $\Phi_{kp^m}(X)$ divides $f(X)$ in $\Z[X]$.
	Thus $\prod_{p\mid d} \prod_{m=1}^{v_p(d)} \Phi_{kp^m}(X)$ divides $f(X)$ in $\Z[X]$.
	Therefore $\omega_d=f(\zeta_k)\in \prod_{p\mid d} \prod_{m=1}^{v_p(d)} \Phi_{kp^m}(\zeta_k)\;\Z[\zeta_k]$.
	By Lemma~\ref{LemaCyclotomicPolynomials}, each $\Phi_{kp^m}(\zeta_k)$ belongs to $p\;\Z[\zeta_n]$. As $d=\prod_{p\mid d} \prod_{m=1}^{v_p(d)} p$ we deduce that $\omega_d\in d\;\Z[\zeta_n]$, as desired.
\end{proof}

For a positive integer $n$ and a subfield $F$ of $\Q(\zeta_n)$, let $\Gamma_F$ denote a set of representatives of equivalence classes of the following equivalence relation defined on $\Z$:
	$$x\sim y \quad \text{ if and only if } \quad \zeta_n^x \text{ and } \zeta_n^y \text{ are conjugate in } \mathbb{Q}(\zeta_n) \text{ over } F.$$

\begin{corollary}\label{corollarydR}
	Let $n$ be a positive integer, let $F$ be a subfield of $\Q(\zeta_n)$ and let $R$ be the ring of integers of $F$.
	For every $x\in \Gamma_F$ let $B_x$ be an integer and for every integer $i$ define
	$$\omega_i = \sum_{x\in \Gamma_F} B_x \Tr_{\Q(\zeta_n)/F}(\zeta_n^{ix}).$$
	Let $d$ be a divisor of $n$ such that $\omega_{\frac{d}{q}}=0$ for every prime power $q$ dividing $d$ with $q\ne 1$.
	Then $\omega_d\in d\;R$.
\end{corollary}

\begin{proof}
Apply Proposition~\ref{NT} to the integers $A_x=B_{\overline{x}}$ with $\overline{x}$ denoting the class in $\Gamma_F$ containing $x$.
\end{proof}

In the remainder of this section we reserve the letter $p$ to denote positive prime integers.

We now introduce some notation for a positive integer $n$ which will be fixed throughout.
First we set
	$$\PP{n} = \prod_{p\mid n} p \quad \text{ and } \quad n_p=p^{v_p(n)}.$$
If moreover $x\in \Z$ then we set
\begin{eqnarray*}
	\MM{x}{n}  &=& \text{ representative of the class of } x \text{ modulo } n \text{ in the interval } \left(-\frac{n}{2}, \frac{n}{2}\right]; \\
	\AMM{x}{n} &=&\text{ the absolute value of } \MM{x}{n} \text{ and;}\\
	\gamma_n(x) &=& \prod_{\substack{p\mid n \\ \AMM{x}{n_p} <\frac{n_p}{2p}}} p.
\end{eqnarray*}

Next lemma collects two elementary properties involving this notation whose proofs are direct consequences of the definitions.

\begin{lemma}\label{Elementary}
	Let $p$ be a prime dividing $n$ and let $x,y \in \Z$.
	Then the following conditions hold:
	\begin{enumerate}
		\item\label{pdividinglevel} If $p\mid \gamma_n(x)$ then $\MM{x}{\frac{n_p}{p}} \equiv x \bmod n_p$.
		\item\label{LevelStep}
		Let $d \mid \PP{n}$ such that $x \equiv y \bmod \frac{n}{d}$. If $d$ divides both $\gamma_n(x)$ and $\gamma_n(y)$ then $x \equiv y \bmod n$.
	\end{enumerate}
\end{lemma}

For integers $x$ and $y$ we define the following equivalence relation on $\Z$:
$$x \sim_n y \quad \Leftrightarrow \quad x \equiv \pm y \mod n.$$
We denote by $\Gamma_n$ a set of representatives of these equivalence classes.
Without loss of generality one may assume that $\Gamma_n=\Gamma_{\Q(\zeta_n+\zeta_n^{-1})}$.

In the remainder of the section we assume that $n$ is odd.
For $x$ and $y$ integers let
$$\alpha^{(n)}_x =  \zeta_n^x + \zeta_n^{-x}, \quad
\kappa^{(n)}_x = \begin{cases}2,&\text{if }x\equiv 0 \bmod n; \\
1,&\text{otherwise};\end{cases}
\quad\text{and}\quad
\delta_{x,y}^{(n)} = \begin{cases}
1, & \text{if } x\sim_n y  ; \\
0, & \text{otherwise}.
\end{cases}
$$
Moreover, $\Q\left(\alpha^{(n)}_1\right) = \Q(\zeta_n+\zeta_n^{-1})$ is the maximal real subfield of $\Q(\zeta_n)$ and $\Z\left[\alpha^{(n)}_1\right]=\Z[\zeta_n+\zeta_n^{-1}]$ is the ring of integers of $\Q\left(\alpha^{(n)}_1\right)$.

Let
    $$\B_n=\left\{x \in  \Z/n\Z : \AMM{x}{n_p}> \frac{n_p}{2p} \text{ for every } p\mid n \right\}
\quad \text{ and }\quad
    \mathcal{B}_n=\{ \alpha^{(n)}_b : b\in \B_n\}.$$

In the following proposition we prove that $\mathcal{B}_n$ is a $\Q$-basis of $\Q[\alpha_1^{(n)}]$.
For $x\in \Q[\alpha_1^{(n)}]$ and $b\in \mathcal{B}_n$, we use
$$C_b(x) = \text{coefficient of } \alpha^{(n)}_b \text{ in the expression of } x \text{ in the basis }\mathcal{B}_n.$$

We denote by $\mu$ the number theoretical M\"obius function.
\begin{proposition}\label{BaseLemma}
Let $n$ be a positive odd integer.
Then
\begin{enumerate}
\item $\mathcal{B}_n$ is a $\Z$-basis of $\Z\left[\alpha^{(n)}_1\right]$ and in particular, a $\Q$-basis of $\Q\left(\alpha^{(n)}_1\right)$.
\item\label{CoefLambda} If $b\in \B_n$ and $i\in \Z$ then
	$C_b(\alpha^{(n)}_i) = \kappa^{(n)}_i \cdot \mu(\gamma(i))\cdot \delta_{b,i}^{(n/\gamma(i))}.$
\end{enumerate}
\end{proposition}

\begin{proof}
It is easy to see that $|\mathcal{B}_n|\leq \frac{\varphi(n)}{2}=[\Q\left(\alpha^{(n)}_1\right):\Q]$. Thus it is enough to prove the following equality

	$$\alpha^{(n)}_i=\kappa^{(n)}_i \mu(\gamma(i)) \sum_{b\in \B_n, b\sim_{\frac{n}{\gamma(i)}} i}  \alpha^{(n)}_b.$$
	Actually we will show
	$\zeta_n^i = \mu(\gamma(i)) \sum_{b \in \B_n,  b\equiv i \bmod \frac{n}{\gamma(i)}}  \zeta_n^b$,
	which implies the desired expression of $\alpha^{(n)}_i$.		
Indeed, for every $p\mid n$ let $\zeta_{n_p}$ denote the $p$-th part of $\zeta_n$, i.e. $\zeta_{n_p}$ is a primitive $n_p$-th root of unity and $\zeta_n = \prod_{p \mid n} \zeta_{n_p}$.
Let $J$ be the set of tuples $(j_p)_{p\mid \gamma(i)}$ satisfying $j_p\in \{1,\dots, p-1\}$ for every $p\mid \gamma(i)$.
For every $j\in J$ let $b_j\in \Z/n\Z$ given by
	$$b_j\equiv \begin{cases}
	i+j_p\frac{n_p}{p} \mod n_p, & \text{if } p\mid \gamma(i); \\
	i \mod n_p, & \text{otherwise}.
	\end{cases}$$
Then $\{b_j : j\in J \}$ is the set of elements $b$ in $\B_n$ satisfying $i\equiv b \bmod \frac{n}{\gamma(i)}$.
	From
	$$0 = \zeta_{n_p}^i \left(1 + \zeta_{n_p}^{\frac{n_p}{p}} + \zeta_{n_p}^{\frac{2n_p}{p}}
    + \dots + \zeta_{n_p}^{\frac{(p-1)n_p}{p}} \right)$$
we obtain $\zeta_{n_p}^i = - \sum_{j_p=1}^{p-1} \zeta_{n_p}^{i+j_p\frac{n_p}{p}}$.
Therefore
	$$\zeta_n^i =
	\prod_{\substack{p \mid n \\ p \nmid \gamma(i)}} \zeta_{n_p}^i \prod_{\substack{p \mid n \\ p \mid \gamma(i)}}
	\left(-\sum_{j_p=1}^{p-1 }  \zeta_{n_p}^{i+j_p\frac{n_p}{p}} \right) =
	\mu(\gamma(i)) \sum_{j\in J} \zeta_n^{b_j} =
	\mu(\gamma(i)) \sum_{\substack{b\equiv i \bmod \frac{n}{\gamma(i)} \\ b\in\B_n} } \zeta_n^b.$$
\end{proof}

\section{Proof of Theorem~\ref{Main}}\label{SectionProofOfMainResult}

In this section we prove Theorem~\ref{Main}.
In the remainder, set $G=\PSL (2,t^f)$ with $t$ a prime.
Our goal is to prove that any element $u$ of order $n$ in $\V(\Z G)$, where $n$ is greater than $1$ and coprime with $2t$, is rationally conjugate to an element of $G$. By Theorem~\ref{GroupProperties}.(\ref{PSL2PrimePower}) we may also assume that $n$ is not a prime power.

As the order $n$ of $u$ is fixed throughout, we simplify the notation of the previous section by setting
$$\gamma=\gamma_n, \quad \alpha_x=\alpha^{(n)}_x, \quad \kappa_x=\kappa_{x}^{(n)}, \quad \B=\B_n, \quad \mathcal{B}=\mathcal{B}_n.$$

We argue by induction on $n$. So we assume that $u^d$ is rationally conjugate to an element of $G$ for every divisor $d$ of $n$ with $d\ne 1$.

We will use the representations $\Theta_m$ and Brauer characters $\psi_m$ introduced in Theorem~\ref{GroupProperties}.(\ref{PSL2ModularCharacters}). As usual in modular representation theory, a bijection between the complex roots of unity of order coprime with $t$ and the roots of unity of the same order in a field of characteristic $t$ has been fixed a priori. In this sense we will identify the eigenvalues of $\Theta_m$ and the summands in $\psi_m$.
Since units of prime order in $\V(\Z G)$ are rationally conjugate to elements of $G$ by Theorem~\ref{GroupProperties}.(\ref{PSL2PrimePower}), we know that the kernel of $\Theta_1$ on $\langle u \rangle$ is trivial and hence  $\Theta_1(u)$ has order $n$.
As the values of $\psi_1$ on $t$-regular elements of $G$ are real, by Theorem~\ref{GroupProperties}.(\ref{PSL2GroupProperties}) and Theorem~\ref{Known}.(\ref{BrauerCharacterOnUnit}), the set of eigenvalues of $\Theta_{1}(u)$ is closed under
taking inverses (counting multiplicities).
Therefore, $\Theta_{1}(u)$ is conjugate to $\diag(1, \zeta, \zeta^{-1})$ for a suitable primitive $n$-th root of unity $\zeta$.
Hence by Theorem~\ref{GroupProperties} there exists an element $g_{0}\in G$ of order $n$ such that $\Theta_{1}(g_{0})$ and $\Theta_{1}(u)$ are conjugate.
From now on we abuse the notation and consider $\zeta$ both as a primitive $n$-th root of unity in a field of characteristic $t$ and as a complex primitive $n$-th root of unity.
Then for any positive integer $m$ we have that
\begin{equation}\label{ThetaDiag}
\Theta_{m}(g_{0})\text{ is conjugate to }
\diag\left(1, \zeta, \zeta^{-1}, \zeta^{2}, \zeta^{-2}, \ldots, \zeta^{m}, \zeta^{-m} \right),
\end{equation}
and for every integer $i$ we have
\begin{equation}\label{ChiExpresion}
\psi_m(g_{0}^{i})=\sum_{j=-m}^{m}\zeta^{ij} = 1 + \sum_{j=1}^m \alpha_{ij}.
\end{equation}
The element $g_0 \in G$ and the primitive $n$-th root of unity $\zeta$ will be fixed throughout.

By Theorem~\ref{GroupProperties}.(\ref{PSL2GroupProperties}), $x\mapsto (g_0^x)^G$ defines a bijection from $\Gamma_n$ to the set of conjugacy classes of $G$ formed by elements of order dividing $n$.
For an integer $x$ (or $x\in\Gamma_n$) we set
$$\varepsilon_x = \varepsilon_{g_0^x}(u) \quad \text{ and }\quad
\lambda_x = \sum_{i\in \Gamma_n} \varepsilon_{i} \alpha_{ix}.$$
By Theorem~\ref{Known}, $u$ is rationally conjugate to an element of $G$ if and only if $\varepsilon_x \ge 0$ for every $x\in \Gamma_n$.

\begin{lemma}\label{LemmaLambdaAlpha}
	$u$ is rationally conjugate to $g_0$ if and only if
	\begin{equation}\label{lambdaequalalpha}
	\lambda_i = \alpha_i, \ \text{for any positive integer } i.
	\end{equation}
\end{lemma}
\begin{proof}
If $u$ is rationally conjugate to $g_0$, then $\varepsilon_1 = 1$ and $\varepsilon_x = 0$ for any $x \in \Gamma_n\setminus \{1\}$.
Therefore (\ref{lambdaequalalpha}) holds.
Conversely, assume that (\ref{lambdaequalalpha}) holds.
For $ v \in \V(\Z G)$ of order dividing $n$ let $\lambda_i'( v) = \sum_{x \in \Gamma_n} \varepsilon_{g_0^x}(v) \alpha_{xi}$.
Then $\lambda_i = \lambda_i'(u) = \sum_{j=0}^{n-1} \varepsilon_{g_0^j}(u) \zeta_n^{ij}$ and $\alpha_i = \sum_{j=0}^{n-1} \varepsilon_{g_0^j}(g_0) \zeta_n^{ij}$.
As the Vandermonde matrix $(\zeta_n^{ij})_{1\le i,j\le n}$ is invertible we deduce that $\varepsilon_{g_0^j}(u) = \varepsilon_{g_0^j}(g_0)$ for every $j \in \Gamma_n$. So $\varepsilon_j=\varepsilon_{g_0^j}(u)=\varepsilon_{g_0^j}(g_0)=0$ for every  $j\in \Gamma_n\setminus \{1\}$ and $\varepsilon_1=1$. As we are assuming that if $d$ is a divisor of $n$ different from $1$ then $u^d$ is rationally conjugate to an element of $G$, we also have $\varepsilon_g(u^d)\geq 0$ for every $g\in G$.
Thus $u$ is rationally conjugate to an element of $g\in G$ by Theorem~\ref{Known}.(\ref{MRSW}).
Then $\varepsilon_{g_0}(g)=\varepsilon_{g_0}(u)=1$ and therefore $g$ is conjugate to $g_0$ in $G$.
We conclude that $u$ and $g_0$ are rationally conjugate.
\end{proof}

By Lemma~\ref{LemmaLambdaAlpha}, in order to achieve our goal it is enough to prove (\ref{lambdaequalalpha}).
We argue by contradiction, so suppose that $\lambda_d\ne \alpha_d$ for some positive integer $d$ which we assume to be minimal with this property.
Observe that if $\lambda_i = \alpha_i$ and $j$ is an integer such that $\gcd(i, n) =
\gcd(j, n)$, then there exists  $\sigma \in \Gal(\Q (\alpha_1)/\Q )$ such that $\sigma(\alpha_i) = \alpha_j$ and applying $\sigma$ to the
equation $\lambda_i = \alpha_i$ we obtain $\lambda_j = \alpha_j$.  	
This implies that $d$ divides $n$. Note that $\alpha_1 = \lambda_1$ by our choice of $g_0$ and hence $d\ne 1$.
Moreover, $d\ne n$ because $\lambda_n = 2 \sum_{x\in \Gamma_n} \varepsilon_x = 2 = \alpha_n$ as the augmentation of $u$ is 1.

We claim that
\begin{equation}\label{KeyEquality}
\lambda_d = \alpha_d + d\tau \text{ for some } \tau \in \Z[\alpha_1].
\end{equation}
Indeed, for any $x\in\Gamma_{n}$ let $B_x=\varepsilon_x -1$ if $x\sim_n 1$ and $B_x = \varepsilon_x$ otherwise. Then for any integer $i$ we have $\lambda_i - \alpha_i = \sum_{x\in\Gamma_n} B_x \Tr_{\Q(\zeta) / \Q(\alpha_1)} \left(\zeta^{ix}\right)$. Therefore, applying Corollary~\ref{corollarydR} for $F=\Q(\alpha_1)$, $R=\Z[\alpha_1]$ and $\omega_i = \lambda_i -\alpha_i$, the claim follows.

By (\ref{ChiExpresion}) we have, using Theorem \ref{Known}.(\ref{BrauerCharacterOnUnit}),
\begin{equation}\label{PsidCharacterValues}
\psi_d(g_0)=1+\sum_{i=1}^d\alpha_{i} \quad \text{and} \quad \psi_d(u) =  \sum_{x\in \Gamma_n} \varepsilon_x \psi_d(g_0^x) = \sum_{x\in \Gamma_n} \varepsilon_x \left(1+\sum_{i=1}^d \alpha_{ix}\right) = 1+\sum_{i=1}^d\lambda_i.
\end{equation}

Combining this with (\ref{KeyEquality}) and the minimality of $d$, we obtain $\psi_d(u)=\psi_d(g_0)+d\tau$.
Furthermore, $\tau\ne 0$, as $\lambda_d\ne \alpha_d$.
Therefore
\begin{equation}\label{dmu}
C_b(\psi_d(u)-1) \equiv  C_b(\psi_d(g_0)-1) \bmod d \quad \text{ for every } b\in \B
\end{equation}
and
\begin{equation}\label{Difference}
d\le \left|C_{b_0}(\psi_d(u)-1) - C_{b_0}(\psi_d(g_0)-1)\right| \quad \text{ for some } b_0\in\B.
\end{equation}

The bulk of our argument relies on an analysis of the eigenvalues of $\Theta_d(u)$ and the induction hypothesis on $n$ and $d$.
More precisely, we will use (\ref{dmu}) and (\ref{Difference}) to obtain a contradiction by comparing the eigenvalues of $\Theta_d(g_0)$ and $\Theta_d(u)$.
Of course we do not know the eigenvalues of the latter but we know the eigenvalues of each $\Theta_d(g_0^i)$.
Moreover, if $c$ is a divisor of $n$ with $c\ne 1$ then $u^c$ is rationally conjugate to an element $g$ of $G$.
Then $\Theta_1(g)$, $\Theta_1(u^c)$ and $\Theta_1(g_0^c)$ are conjugate in $M_3(F)$, for a suitable field $F$, and as $\Theta_1$ is injective on $\GEN{g_0}$ and $g$ is conjugate to an element of $\GEN{g_0}$ we conclude that $u^c$ is conjugate to $g_0^c$. Thus we know the eigenvalues of $\Theta_d(u^c)$. This has consequences for the eigenvalues of $\Theta_d(u)$.

To be more precise we fix $\nu_1,\dots,\nu_d\in \Gamma_n$ (with repetitions if needed) such that the eigenvalues of $\Theta_d(u)$ with multiplicities are $1,\zeta^{\pm \nu_1},\dots,$ $\zeta^{\pm \nu_d}$.
This is possible by the last statement of Theorem~\ref{GroupProperties}.\eqref{PSL2GroupProperties}.
By the above paragraph, if $c\mid n$ with $c\ne 1$ then the lists $(c\nu_i)_{1\le i\le d}$ and $(ci)_{1\le i\le d}$ represent the same elements in $\Gamma_n$, up to ordering, and hence
$(\nu_i)_{1\le i\le d}$ and $(i)_{1\le i\le d}$ represent the same elements of $\Gamma_{\frac{n}{c}}$, up to ordering.
We express this by writing
$$(\nu_i)\sim_{\frac{n}{c}} (i) \quad \text{ for every } c \mid n \text{ with } c\ne 1.$$
This provides restrictions on $d$, $n$ and the $\nu_i$.

Moreover, $C_b(\psi_d(u)-1)$ and $C_b(\psi_d(g_0)-1)$ are the coefficients of $\alpha_b$ in the expression in the basis $\mathcal{B}$ of $\alpha_{\nu_1}+\dots+\alpha_{\nu_d}$ and $\alpha_1+\dots+\alpha_d$, respectively.
By \eqref{PsidCharacterValues} and Proposition~\ref{BaseLemma} we obtain for every $b\in\B$ that
\begin{equation}\label{Separated}
C_b(\psi_d(g_0)-1) = \sum_{i=1}^d \mu(\gamma(i))\cdot \delta_{b,i}^{(n/\gamma(i))} \quad \text{ and } \quad C_b(\psi_d(u)-1) = \sum_{i=1}^d \kappa_{\nu_i}\cdot\mu(\gamma(\nu_i))\cdot\delta_{b,\nu_i}^{(n/\gamma(\nu_i))}.
\end{equation}
and so
\begin{equation}\label{cbchiuMenos1}
C_b(\psi_d(u)-1) - C_b(\psi_d(g_0)-1) = \sum_{i=1}^d \left(\kappa_{\nu_i}\cdot\mu(\gamma(\nu_i))\cdot\delta_{b,\nu_i}^{(n/\gamma(\nu_i))} - \mu(\gamma(i))\cdot\delta_{b,i}^{(n/\gamma(i))}\right).
\end{equation}

\begin{lemma}\label{KappaAndPrime}
	\begin{enumerate}
		\item\label{kappa1} If $\kappa_{\nu_i}\ne 1$ for some $1\le i\le d$ then $\frac{n}{d}$ is the smallest prime dividing $n$ and $\kappa_{\nu_j}=1$ for every $1\le j \le d$ with $j\ne i$.
		\item\label{PrimeGreaterThand} If $d>3$ then $n$ is not divisible by any prime greater than $d$.
	\end{enumerate}
\end{lemma}

\begin{proof}
	Let $p$ denote the smallest prime dividing $n$.
	
	\eqref{kappa1} Suppose that $\kappa_{\nu_i}\ne 1$. Then $\nu_i\equiv 0 \bmod n$.
	As $(i)\sim_{\frac{n}{p}} (\nu_i)$ we deduce that $k\equiv 0 \bmod \frac{n}{p}$ for some $1\le k\le d$. Therefore $d=k=\frac{n}{p}$ and for every $1\le j\le d$ with $j\ne i$ we have $\nu_j\not\equiv 0 \bmod \frac{n}{p}$.  Hence $\kappa_{\nu_j}=1$.
	
	\eqref{PrimeGreaterThand} Suppose that $q$ is a prime divisor of $n$ with $d<q$.
	Then $\frac{n}{d}\ne p$ and therefore, by \eqref{kappa1}, $\kappa_{\nu_i}=1$ for every $1\le i \le d$.
	Thus, by \eqref{Difference} and (\ref{cbchiuMenos1}) and ignoring the signs provided by the $\mu(\gamma(i))$ and $\mu(\gamma(\nu_i))$, it is enough to show that $\delta^{(n/\gamma(i))}_{b,i}\ne 0$ for at most two $i$'s and
	$\delta^{(n/\gamma(\nu_i))}_{b,\nu_i}\ne 0$ for at most two $i$'s, since by assumption $d > 3$, i.e. $d \geq 5$.
	Observe that if $1\le i\le d$ then $q\nmid i$ and hence $\frac{n}{\gamma(i)}$ is multiple of $q$.
	Moreover, if $1\le i, j \le d$ with $i\ne j$  then $-q<i-j<i+j<2q$. Therefore $i\not\sim_q j$ unless $j=q-i$.
	As $(i) \sim_{n/p} (\nu_i)$ and $q \mid \frac{n}{p}$ we have $(i)\sim_q (\nu_i)$, the lemma follows.
\end{proof}

For a non-zero integer $m$ let $P(m)$ denote the number of prime divisors of $m$.
We obtain an upper bound for $\left|C_b(\psi_d(u)-1) - C_b(\psi_d(g_0)-1)\right|$ in terms of $P(d)$.

\begin{lemma}\label{BoundDifference}
	For every $b\in\B$ we have
	$$\left|C_b(\psi_d(u)-1) - C_b(\psi_d(g_0)-1)\right| \le 1+2^{P(d)+2}.$$
	Moreover if $\kappa_{\nu_i} = 1$ for every $1\leq i \leq d$ then
	$$\left|C_b(\psi_d(u)-1) - C_b(\psi_d(g_0)-1)\right| \le 2^{P(d)+2}.$$
\end{lemma}

\begin{proof}
	Using (\ref{cbchiuMenos1}), and ignoring the sings given by $\mu(\gamma(i))$ and $\mu(\gamma(\nu_i))$, it is enough to prove that
	$$\sum_{i=1}^d \delta_{b,i}^{(n/\gamma(i))}\le 2^{P(d)+1} \quad \text{ and } \quad \sum_{i=1}^d \kappa_{\nu_i}\delta_{b,\nu_i}^{(n/\gamma(\nu_i))}\le 1+2^{P(d)+1}.$$

	Observe that $\kappa_{\nu_i}=2$ for at most one $i$ by Lemma~\ref{KappaAndPrime}.(\ref{kappa1}).
	Recall that $d' = \prod_{p \mid d}p$. Thus the lemma is a consequence of the following inequalities for every $e$ dividing $\PP{d}$:
	$$\left| \left\{  1\le i\le d : \gcd(d,\gamma(i))=e, \delta_{b,i}^{(n / \gamma(i))}=1 \right\} \right|  \leq  2 \quad \text{ and}$$

	$$ 	\left| \left\{  1\le i\le d : \gcd(d,\gamma(\nu_i))=e,
	\delta_{b,\nu_i}^{(n / \gamma(\nu_i))}=1 \right\} \right| \leq  2,
	$$
	since the number of divisors of $d'$ is $2^{P(d)}$ and if $\kappa_{\nu_i} = 2$ for some $\nu_i$ this provides an additional $1$.
	We prove the second inequality, only  using that $(\nu_i)\sim_d (i)$.
	This implies the first inequality by applying the second one to $u=g_0$.
	
	For a fixed $e$ dividing $d'$ let $Y_e = \left\{ 1\le i \le d : \gcd(d,\gamma(\nu_i))=e, \delta_{b,\nu_i}^{(n / \gamma(\nu_i))}=1 \right\}$.
	By changing the sign of some $\nu_i$'s, we may assume without loss of generality that if
	$\delta_{b,\nu_i}^{(n/\gamma(\nu_i))}=1$ then $b\equiv \nu_i \bmod \frac{n}{\gamma(\nu_i)}$.
	Thus, if $i\in Y_e$ then $b\equiv \nu_i \bmod \frac{n}{\gamma(\nu_i)}$.
	We claim that if $i,j\in Y_e$ then $\nu_i\equiv \nu_j \bmod d$.
	Indeed, let $p$ be prime divisor of $d$. If $n_p \neq d_p$ then $d_p \leq \left(\frac{n}{\gamma(\nu_i)}\right)_p$, so $\nu_i\equiv \nu_j \bmod d_p$. If $p \nmid e$ then $n_p =  \left(\frac{n}{\gamma(\nu_i)}\right)_p$ and so also $\nu_i\equiv \nu_j \bmod d_p$.
	Otherwise, i.e. if $n_p=d_p$ and $p\mid e$, then $p$ divides both $\gamma(\nu_i)$ and $\gamma(\nu_j)$ and $\nu_i \equiv \nu_j \bmod \frac{d_p}{p}$. Therefore $\nu_i\equiv \nu_j \bmod n_p$, by Lemma~\ref{Elementary}.\eqref{LevelStep}.
	As $(\nu_i)\sim_d (i)$ and there are at most two $i$'s with $1\le i\le d$ representing the same class in $\Gamma_d$, we deduce that $|Y_e|\le 2$, as desired.
\end{proof}

We are ready to finish the proof of Theorem~\ref{Main}. Recall that we are arguing by contradiction and $n$, and hence also $d$, is odd.

By (\ref{Difference}) and Lemma~\ref{BoundDifference} we have $d\le 1+2^{P(d)+2}$ and this has strong consequences on the possible values of $d$. 	
Indeed if $P(d)\ge 3$ then
$$1+2^{P(d)+2}\ge d \ge 3\cdot 5\cdot 7\cdot 2^{P(d)-3} > (105-2^5) +  2^{P(d)+2} = 73 + 2^{P(d)+2},$$
 a contradiction. Thus, if $P(d)=2$ then $d=15$ and if $P(d)=1$ then $d\in \{3,5,7,9\}$.

However, if $d=9$ then $\left|C_{b_0}(\psi_9(u)-1) - C_{b_0}(\psi_9(g_0)-1)\right|=9$ by Lemma~\ref{BoundDifference} and hence $\kappa_{\nu_i}=2$ for one $1\le i\le d$. This implies, by Lemma~\ref{KappaAndPrime}.(\ref{kappa1}), that $n=27$ contradicting the assumptions that $n$ is not a prime power. Therefore $d\in \{3,5,7,15\}$.
We deal with these cases separately using \eqref{Separated} and \eqref{cbchiuMenos1}.
Observe that if $p$ is a prime bigger than $d$ then $p \mid \frac{n}{\gamma(i)}$ for every $1\leq i \leq d$ and so also $p \mid \frac{n}{\gamma(\nu_i)}$, since $(i) \sim_p (\nu_i)$.

\underline{Assume that $d = 3$}. Combining Lemma~\ref{KappaAndPrime}.(\ref{kappa1}) with the assumptions that $n$ is not a prime power, we deduce that $\kappa_{\nu_i}=1$ for every $1\le i \le 3$.
Suppose that there is a prime $p\mid n$ with $p\ge 7$.
Then $p\mid \frac{n}{\gamma(i)}$ and $p\mid \frac{n}{\gamma(\nu_i)}$ for every $1\le i \le 3$.
Thus
$$\left|\left\{1\le i \le 3 : \delta_{b,i}^{(n/\gamma(i))} =1\right\}\right|\le 1 \text{ and } \left|\left\{1\le i \le 3 : \delta_{b,\nu_i}^{(n/\gamma(\nu_i))} =1\right\}\right|\le 1 \text{ for every } b\in\B$$
which implies $|C_{b_0}(\psi_3(u)-1) - C_{b_0}(\psi_3(g_0)-1)|\le 2$, contradicting (\ref{Difference}).
So $\PP{n}=15$.

Moreover, $n_3=3$ because otherwise $3\mid \frac{n}{\gamma(i)}$ and $3\mid \frac{n}{\gamma(\nu_i)}$ and so $15\mid \frac{n}{\gamma(i)}$ and $15\mid \frac{n}{\gamma(\nu_i)}$ for every $1\le i \le 3$. Hence $|C_{b_0}(\psi_3(u)-1)|$ and $|C_{b_0}(\psi_3(u)-1)|$ are both at most 1, in contradiction with \eqref{Difference}.
If $5^3\mid n$, then $25 \mid \frac{n}{\gamma(\nu_i)}$ and $25 \mid \frac{n}{\gamma(i)}$ for every $1\le i\le 3$, which implies
$\left| C_{b_0}(\psi_3(g_0)-1) \right| \leq 1$ and $\left| C_{b_0}(\psi_3(u)-1) \right| \leq 1$, again a contradiction.
Therefore $n\in \{15,75\}$.
Since $(i) \sim_3 (\nu_i)$, we may assume that $3\mid \nu_3$ and $3\nmid \nu_i$ for $i=1,2$.

Suppose that $n = 15$. Then, as $(i)\sim_5 (\nu_i)$, we have $\gamma(1) = \gamma(2) = \gamma(\nu_1) = \gamma(\nu_2) = 1$ and $\gamma(3) = \gamma(\nu_3) = 3$. So
$$C_b(\psi_3(g_0)-1) = \delta_{b,1}^{15} + \delta_{b,2}^{15} -\delta_{b,3}^{5} \text{ and } C_b(\psi_3(u)-1) = \delta_{b,\nu_1}^{15} + \delta_{b,\nu_2}^{15} -\delta_{b,\nu_3}^{5} \text{ for every }b\in\B,$$
implying
$$ C_b(\psi_3(u)-1) - C_b(\psi_3(g_0)-1) = \delta_{b,\nu_1}^{15} + \delta_{b,\nu_2}^{15} -\delta_{b,\nu_3}^{5} - \delta_{b,1}^{15} - \delta_{b,2}^{15} + \delta_{b,3}^{5}.$$
Since $1 \nsim_{15} 2$, we must have $C_{b_0}(\psi_3(u)-1)- C_{b_0}(\psi_3(g_0)-1) = 3$ and $\nu_3 \sim_5 1$ while $\nu_1 \sim_5 \nu_2 \sim_5 2$.
Then $C_1(\psi_3(u)-1) - C_1(\psi_3(g_0)-1) = -2$, contradicting (\ref{dmu}).

Suppose that $n = 75$. Then $\gamma(1) = \gamma(2) = 5$, $\gamma(3) = 3$ and
$$C_b(\psi_3(g_0)-1) = -\delta_{b,1}^{15} - \delta_{b,2}^{15} - \delta_{b,3}^{25} \text{ for every }b\in \B.$$
Suppose $\nu_3 \sim_{25} 3$.
Then
$$C_b(\psi_3(u)-1) = -\delta_{b,\nu_1}^{15} - \delta_{b,\nu_2}^{15} - \delta_{b,\nu_3}^{25} \text{ for every }b\in \B.$$
As $\delta_{b,\nu_3}^{25}=\delta_{b,3}^{25}$, we have  $|C_{b_0}(\psi_3(u)-1) - C_{b_0}(\psi_3(g_0)-1)| \leq 2$, contradicting (\ref{Difference}).
Thus $\nu_3\not\sim_{25} 3$ and we may assume $\nu_1 \sim_{25} 3$.
If $\nu_3 \sim_{25} 2$ then
$$C_b(\psi_3(u)-1) = \delta_{b,\nu_1}^{75} - \delta_{b,\nu_2}^{15} + \delta_{b,\nu_3}^{5} \text{ for every }b\in \B.$$
However $C_{13}(\psi_3(u)-1)-C_{13}(\psi_3(g_0)-1)=2$, contradicting (\ref{dmu}).
So $\nu_3 \sim_{25} 1$ and arguing as above we obtain $C_{14}(\psi_3(u)-1)-C_{14}(\psi_3(g_0)-1)\in \{1,2\}$, again a contradiction with (\ref{dmu}).

\underline{Assume that $d = 5$}. By Lemma~\ref{KappaAndPrime}.(\ref{PrimeGreaterThand}) and the assumptions on $n$, we obtain $\PP{n}=15$. As $(i)\sim_5 (v_i)$, we may assume that $5\mid \nu_5$ and $5\nmid \nu_i$ for every $1\le i \le 4$.
Suppose that $n = 15$.
In this case
$$C_b(\psi_5(g_0)-1) = \delta_{b,1}^{15} + \delta_{b,2}^{15} - \delta_{b,3}^{5} + \delta_{b,4}^{15} - \delta_{b,5}^{3}\text{ for every }b\in \B.$$
If $3\mid \nu_5$ and $3\nmid \nu_i$ for every $1\le i\le 4$, then
$$C_b(\psi_5(u)-1) =  \delta_{b,\nu_1}^{15} + \delta_{b,\nu_2}^{15} + \delta_{b,\nu_3}^{15} + \delta_{b,\nu_4}^{15} + 2\text{ for every }b\in \B$$
and hence
$$C_1(\psi_5(u)-1)-C_1(\psi_5(g_0)-1)=2+\delta_{1,\nu_1}^{15} + \delta_{1,\nu_2}^{15} + \delta_{1,\nu_3}^{15} + \delta_{1,\nu_4}^{15} \le 4,$$ contradicting (\ref{dmu}).
Therefore, as $(i)\sim_3 (\nu_i)$, we may assume that $3\mid \nu_1$ and $3\nmid \nu_i$ for every $2\le i \le 5$. This implies
$$C_b(\psi_5(u)-1) = - \delta_{b,\nu_1}^{5} + \delta_{b,\nu_2}^{15} + \delta_{b,\nu_3}^{15} + \delta_{b,\nu_4}^{15} -\delta_{b,\nu_5}^3\text{ for every }b\in \B.$$
As both $\left|C_{b_0}(\psi_5(u)-1)\right|$ and $\left|C_{b_0}(\psi_5(g_0)-1)\right|$ are at most $2$, we obtain a contradiction with (\ref{Difference}).
Therefore $n\ne 15$ and $\kappa_{\nu_i}=1$ for every  $1\le i \le 5$ by Lemma~\ref{KappaAndPrime}.(\ref{kappa1}).

If $25\mid n$ or $27\mid n$ then it is easy to see that $\left| C_{b_0}(\psi_5(u)-1) \right|\leq 2$ and $\left| C_{b_0}(\psi_5(g_0)-1) \right| \leq 2$, contradicting (\ref{Difference}).
Thus $n=45$.
In this case we have
$$C_b(\psi_5(g_0)-1) = -\delta_{b,1}^{15} + \delta_{b,2}^{45} + \delta_{b,3}^{45} + \delta_{b,4}^{45} - \delta_{b,5}^{9}\text{ for every }b\in\B.$$
If $\nu_5 \sim_9 1$ then
$$C_b(\psi_5(u)-1) = \delta_{b,\nu_1}^{45} + \delta_{b,\nu_2}^{45} + \delta_{b,\nu_3}^{45} + \delta_{b,\nu_4}^{45} + \delta_{b,\nu_5}^3\text{ for every }b\in\B.$$
As $(i)\sim_{15} (\nu_i)$, we obtain $\left| C_{b_0}(\psi_5(u)-1) \right|\le 2$ and $\left| C_{b_0}(\psi_5(g_0)-1) \right|\le 2$ contradicting (\ref{Difference}).
If $\nu_5\not\sim_9 1$ then we may assume that $\nu_1 \sim _9 1$.
Hence
$$C_b(\psi_5(u)-1) = -\delta_{b,\nu_1}^{15} + \delta_{b,\nu_2}^{45} + \delta_{b,\nu_3}^{45} + \delta_{b,\nu_4}^{45} - \delta_{b,\nu_5}^{9}\text{ for every }b\in\B.$$
Again as $(i)\sim_{15}(\nu_i)$, we have both $\left| C_{b_0}(\psi_5(u)-1) \right|$ and $\left| C_{b_0}(\psi_5(g_0)-1) \right|$ at most $2$, which yields a contradiction.

\underline{Assume that $d=7$.}
As $(i)\sim_7 (\nu_i)$, we may assume that $7\mid \nu_7$ and $7\nmid \nu_i$ for every $1\le i \le 6$.
Thus $7\mid \frac{n}{\gamma(i)}$ and $7\mid \frac{n}{\gamma(\nu_i)}$ for every $1\le i \le 6$.
Hence $\left|C_{b_0}(\psi_7(g_0)-1)\right|\le 3$. Moreover, if $\kappa_{\nu_7}\ne 2$ the we also have $\left|C_{b_0}(\psi_7(u)-1)\right|\le 3$ yielding a contradiction with (\ref{Difference}).
Therefore $\kappa_{\nu_7}=2$ (i.e. $n\mid \nu_7$) and by Lemma~\ref{KappaAndPrime} and the assumptions on $n$ we deduce that either $n=21$ or $n=35$.

Suppose that $n=21$.
As $(i)\sim_3 (\nu_i)$ we may assume that  $3\mid \nu_3$ and $3\nmid \nu_i$ for every $i\in \{1,2,4,5,6\}$.
This implies for every $b\in\B$ that
$$C_{b}(\psi_7(u)-1) = \delta_{b,\nu_1}^{21} +\delta_{b,\nu_2}^{21}-\delta_{b,\nu_3}^{7}+\delta_{b,\nu_4}^{21}+\delta_{b,\nu_5}^{21}+\delta_{b,\nu_6}^{21}+2$$
and
$$C_{b}(\psi_7(g_0)-1) = \delta_{b,1}^{21} +\delta_{b,2}^{21}-\delta_{b,3}^{7}+\delta_{b,4}^{21}+\delta_{b,5}^{21}-\delta_{b,6}^{7}-\delta_{b,7}^{3}.$$
Hence, as $(i)\sim_7 (\nu_i)$, we obtain $\left|C_{b_0}(\psi_7(g_0)-1)\right|\le 2$ and $\left|C_{b_0}(\psi_7(u)-1)\right|\le 4$, contradicting (\ref{Difference}).

Suppose that $n=35$.
As $(i)\sim_5 (\nu_i)$ and $(i)\sim_7(\nu_i)$, we have for every $b\in\B$ that
$$C_{b}(\psi_7(u)-1) = \delta_{b,\nu_1}^{35} +\delta_{b,\nu_2}^{35}+\delta_{b,\nu_3}^{35}+\delta_{b,\nu_4}^{35}+\delta_{b,\nu_5}^{35}+\delta_{b,\nu_6}^{35}+2$$
and
$$C_{b}(\psi_7(g_0)-1) = \delta_{b,1}^{35} +\delta_{b,2}^{35}+\delta_{b,3}^{35}+\delta_{b,4}^{35}-\delta_{b,5}^{7}+\delta_{b,6}^{35}-\delta_{b,7}^{5}.$$
Hence, again $\left|C_{b_0}(\psi_7(g_0)-1)\right|\le 2$ and $\left|C_{b_0}(\psi_7(u)-1)\right|\le 4$, yielding a contradiction with (\ref{Difference}).

\underline{Finally assume that $d = 15$}.
Suppose that $n=45$.
In this case we have for every $b\in\B$ that
$$C_b(\psi_{15}(g_0)-1) = -\delta_{b,1}^{15}
+ \delta_{b,2}^{45} + \delta_{b,3}^{45} + \delta_{b,4}^{45} - \delta_{b,5}^{9} + \delta_{b,6}^{45} + \delta_{b,7}^{45} - \delta_{b,8}^{15} - \delta_{b,9}^{15} + \delta_{b,10}^{3} + \delta_{b,11}^{45} + \delta_{b,12}^{45} + \delta_{b,13}^{45} + \delta_{b,14}^{45} - \delta_{b,15}^{9},$$
which implies that $\left|C_{b_0}(\psi_{15}(g_0)-1)\right| \le 4$.
Since $(i) \sim_{5} (\nu_i)$, we deduce that $|C_{b_0}(\psi_{15}(u)-1)| \le 10$, since at most ten of the $\mu(\gamma(\nu_i))$ are equal.
This yields a contradiction with (\ref{Difference}).
Therefore $n\ne 45$ and $\kappa_{\nu_i}=1$ for every $1\le i \le 15$ by Lemma~\ref{KappaAndPrime}.(\ref{kappa1}).
If there is a prime $p\mid n$ with $p\ge 7$ then it is easy to see that $|C_{b_0}(\psi_{15}(g_0)-1)| \leq 7$ and $|C_{b_0}(\psi_{15}(u)-1)| \leq 7$, in contradiction with \eqref{Difference}.
Thus $n' = 15$.
If $25\mid n$ or $27\mid n$ then $|C_{b_0}(\psi_{15}(g_0)-1)| \leq 6$ and $|C_{b_0}(\psi_{15}(u)-1)| \leq 6$, again a contradiction.
As $15$ is a proper divisor of $n$, this implies $n=45$ yielding the final contradiction.

\bibliographystyle{alpha}
\bibliography{PSLOdd}
\end{document}